\documentclass[11pt]{amsart}
\usepackage{amsmath}
\usepackage{mathrsfs}
\usepackage{amssymb, verbatim}
\usepackage{pstricks,pst-node,pst-plot}
\usepackage{tikz, tikz-cd}
\usepackage[all,cmtip]{xy}
\usepackage{comment}
\usepackage{color}

\title[SYZ mirror symmetry for some non-compact CY surfaces]{Recent progress on SYZ mirror symmetry for some non-compact Calabi-Yau surfaces}
\author[T. C. Collins]{Tristan C. Collins}
 \email{tristanc@mit.edu}
  \address{Department of Mathematics, Massachusetts Institute of Technology, 77 Massachusetts Avenue, Cambridge, MA 02139}
 \thanks{T.C.C is supported in part by NSF CAREER grant DMS-1944952 and an Alfred P. Sloan Fellowship. }

  \author[Y.-S. Lin] {Yu-Shen Lin}
   \email{yslin@bu.edu}
  \address{Department of Mathematics, Boston University, 11 Cummington Mall, Boston, MA 02215}
    \thanks{Y.-S. L. is supported in by Simons collaboration grant \# 635846 and NSF grant DMS-2204109.}

\theoremstyle{plain}
\newtheorem{thm}{Theorem}[section]
\newtheorem{prop}[thm]{Proposition}
\newtheorem{defn}[thm]{Definition}

\newtheorem{cor}[thm]{Corollary}
\newtheorem{conj}[thm]{Conjecture}
\newtheorem{que}[thm]{Question}

\theoremstyle{definition}

\newtheorem{rk}[thm]{Remark}

\numberwithin{equation}{section}

\newcommand{\del}{\partial}

\newcommand{\dbar}{\overline{\del}}
\newcommand{\ddb}{\sqrt{-1}\del\dbar}

\newcommand{\be}{\begin{equation}}
\newcommand{\bea}{\begin{eqnarray}}
\newcommand{\eea}{\end{eqnarray}}
 \newcommand{\ee}{\end{equation}}

\renewcommand{\leq}{\leqslant}
\renewcommand{\geq}{\geqslant}

\renewcommand{\epsilon}{\varepsilon}
\renewcommand{\phi}{\varphi}

\begin{document}

\begin{abstract}
We survey the authors recent works, joint with A. Jacob, on Strominger-Yau-Zaslow mirror symmetry for rational elliptic surfaces and del Pezzo surfaces.  We discuss some applications, including the Torelli theorem for $ALH^*$ gravitational instantons and explain how our results can be used to prove that all $ALH^*$ gravitational instantons can be compactified to a weak del Pezzo surface, recovering a recent result of Hein-Sun-Viaclovsky-Zhang.
	 \end{abstract}
\maketitle
\section{Introduction}
The authors and A. Jacob recently undertook a detailed study of aspects of the Strominger-Yau-Zaslow mirror symmetry conjecture for certain log Calabi-Yau surfaces arising from del Pezzo surfaces and rational elliptic surfaces \cite{CJL,CJL2,CJL3}.  The goal of this survey is to recount and summarize some of these results and to discuss some applications and further questions.  

Let us begin by recalling a general formulation of the SYZ mirror symmetry conjecture: 

\begin{conj}[Strominger-Yau-Zaslow]\label{conj: introSYZ}
Let $(\check{X},\check{\omega})$ be a Calabi-Yau manifold, and $\check{\mathcal{M}}_{cplx}$ denote the moduli space of complex structures on $\check{X}$.  Then, for a complex structure $\check{J}\in \check{\mathcal{M}}_{cplx}$ sufficiently close to a large complex structure limit, the following is true:
\begin{enumerate}
\item $(\check{X},\check{J},\check{\omega})$ admits a special Lagrangian torus fibration $\check{\pi}:\check{X}\rightarrow \check{B}$ onto a base $\check{B}$ equipped with an integral affine structure.
\item There is another Calabi-Yau manifold $(X, J, \omega)$ with a special Lagrangian fibration $\pi: X\rightarrow B$ and $B$ is equipped with an integral affine structure.
\item Let $\mathcal{M}_{\text{K\"ah}}$ denote the complexified K\"ahler moduli space of $X$.  There is a mirror map $q: \check{\mathcal{M}}_{cplx}\rightarrow \mathcal{M}_{\text{K\"ah}}$ which is a local diffeomorphism such that ${\rm Im}(q(\check{J})) = \omega$.
\item There is an isomorphism $\phi: \check{B} \rightarrow B$ exchanging the complex and symplectic affine structures, and such that the Riemannian volumes of the special Lagrangian torus fibers over $\check{b} \in \check{B}, \phi(\check{b})\in B$ are inverse to one another.
\end{enumerate}
\end{conj}

This formulation of the SYZ conjecture is perhaps the most optimistic and ambitious version of the conjecture.  Notice that if $X$ is a compact Calabi-Yau of dimension $n$, then the mirror $\check{X}$ is also a compact Calabi-Yau of dimension $n$, and the third point implies the exchange of Hodge numbers
\[
h^{n-1,1}(\check{X}) = h^{1,1}(X).
\]
An important philosophical point for our purposes is that the space $\mathcal{M}_{\text{K\"ah}}$ should be viewed as the (complexified) symplectic moduli space of Calabi-Yau structures on $X$, while  $\check{\mathcal{M}}_{cplx}$ should be viewed as the complex moduli space of Calabi-Yau structures on $\check{X}$.  This recasting of the moduli spaces in terms of Calabi-Yau structures is thanks to Yau's solution of the Calabi conjecture \cite{Y} and the Bogomolov-Tian-Todorov theorem \cite{Bog, Tian, Tod}.   Finally, we remark that work of Hitchin \cite{Hit} shows that the integral affine structure on the base of the special Lagrangian torus fibrations in Conjecture~\ref{conj: introSYZ} is inherited from the complex/symplectic geometry of $X,\check{X}$ and that point (4) of Conjecture~\ref{conj: introSYZ} can be viewed as a consequence of the central principle of the SYZ conjecture: mirror symmetry is $T$-duality.

The authors and A. Jacob recently proved Conjecture~\ref{conj: introSYZ} for certain complete, non-compact Calabi-Yau manifolds \cite{CJL,CJL2}.  Let us give an informal overview of these results.  Let us first introduce the objects relevant for this article:
\begin{itemize}
\item  Let $Y$ be a del Pezzo surface with $K_{Y}^2=d$ where $0 < d \leq 9$, and let $D\in |-K_{Y}|$ be a smooth elliptic curve.  Then a foundational result of Tian-Yau \cite{TY} says that $X=Y\setminus D$ admits an exact, complete, K\"ahler Ricci-flat metric $\omega_{TY} = \ddb \phi_{TY}$ asymptotic to the {\bf Calabi ansatz}.  
\item Let $\check{Y}\rightarrow \mathbb{P}^1$ be a rational elliptic surface (see Definition~\ref{defn: RES}), and suppose $\check{D} \in |-K_{Y}|$ is a singular fiber of Kodaira type $I_{d}$.  Then a result of Hein \cite{Hein} says that $\check{X}= \check{Y}\setminus \check{D}$ admits many complete, K\"ahler, Ricci-flat metrics asymptotic to {\bf semi-flat Calabi-Yau metrics}.
\end{itemize}

In order to prove Conjecture~\ref{conj: introSYZ} one needs to address several geometric questions:
\begin{enumerate}
\item Do the non-compact Calabi-Yau spaces $X, \check{X}$ admit special Lagrangian fibrations?
\item There is no general uniqueness theorem for complete Calabi-Yau metrics, so what is meant by the K\"aher moduli space $\mathcal{M}_{\text{K\"ah}}$?  Can one construct a finite dimensional moduli space parametrizing complete, K\"ahler, Ricci-flat metrics?
\item Assuming we can define $\mathcal{M}_{\text{K\"ah}}$, how do we define the mirror map from the complex moduli of $X$ to the K\"ahler moduli of $\check{X}$?
\end{enumerate}

Roughly speaking, \cite{CJL} addresses question $(1)$, while \cite{CJL2} addresses questions $(2)$ and $(3)$.  In this survey we will recount these results, explain some applications, and highlight some open questions.

\section{Special Lagrangian fibrations in del Pezzo Surfaces} \label{sec: SYZ in dP}
In this section we will describe the construction of special Lagrangian fibrations in del Pezzo surfaces, following the arguments of the authors and A. Jacob in \cite{CJL}.
 
  \subsection{The Calabi ansatz} \label{sec: Calabi ansatz}
    Let $D$ be a projective Calabi-Yau manifold of dimension $(n-1)$ with $K_{D} \simeq \mathcal{O}_{D}$ and let $p: E\rightarrow D$ be a positive line bundle. Denote by $Y_{\mathcal{C}}$ a tubular neighborhood of the zero section in the total space of $E$ and by $X_{\mathcal{C}} := Y_{\mathcal{C}}\setminus\{0\}$ the complement of the zero section in $Y_{\mathcal{C}}$. Let $\Omega_{D}$ denote the holomorphic volume form on $D$. Then $X_{\mathcal{C}}$ admits a holomorphic volume form given by
    	  \begin{align*}
    	    \Omega_{\mathcal{C}}=p^*\Omega_{D}\wedge \frac{dw}{w},
    	  \end{align*} 
	  where $w$ is a coordinate along the fibre of $E$. Let $\omega_D$ be the unique Ricci-flat metric on $D$ such that $[\omega_D]=2\pi c_1(E)$, and normalize $\Omega_{D}$ so that
	  \[
	  \omega_{D}^{n-1} = \frac{1}{2} (\sqrt{-1})^{(n-1)^2}\Omega_{D}\wedge \overline{\Omega}_{D}.
	  \]
	   and let $h$ be a hermitian metric on $E$ such that $\omega_D=-\sqrt{-1}\partial \bar{\partial}\log h$. Consider the K\"ahler form
    	   \begin{align*}
    	     \omega_{\mathcal{C}}=\sqrt{-1}\partial \bar{\partial}\frac{n}{n+1}\big(-\log{|\xi|^2_h}\big)^{\frac{n+1}{n}}
    	   \end{align*} 
	   where $\xi \in E$.  By direct calculation we have
	   \[
	   \omega_{\mathcal{C}}^n = \frac{1}{2}(\sqrt{-1})^{n^2} \Omega_{\mathcal{C}} \wedge \overline{\Omega}_{\mathcal{C}}
	   \]
	    Geometrically, if we let $r(x)$ denote the distance from $x$ to a fixed point $x_0\in X_{\mathcal{C}}$, then a direct calculation shows that 
    	     \begin{equation}\label{eq: geomTY}
    	         |\nabla^k Rm|\leq C_k r^{-\frac{k+2}{n+1}}, \qquad     {\rm inj}(x) \sim r^{\frac{1-n}{n+1}}, \qquad {\rm Vol}(B_{r}(x_0)) \sim r^{\frac{2n}{n+1}}.
    	     \end{equation}
	     In other words, the Riemannian curvature decays at a polynomial rate, but the injectivity radius degenerates at infinity. Thus $X_{\mathcal{C}}$ is not of ``bounded geometry".  Furthermore, it is easy to check that the metric is complete as $|\xi|\rightarrow 0$ but incomplete as $|\xi|\rightarrow 1$. 
    	     
    In their foundational work \cite{TY}, Tian-Yau used the Calabi ansatz $(X_{\mathcal{C}},\omega_{\mathcal{C}},\Omega_{\mathcal{C}})$ to construct complete, non-compact K\"ahler, Ricci-flat manifolds.  Precisely,  given a Fano manifold $Y$ of dimension $n$ and $D$ a smooth anti-canonical divisor, the complement $X=Y\setminus D$ admits a holomorphic volume form $\Omega$ with a simple pole on $D$, and $\Omega$ is unique up to $\mathbb{C}^*$-scaling.  Clearly a neighborhood of $D\subset Y$ and be smoothly identified with a neighborhood of the zero section in the normal bundle to $D$.  Tian-Yau developed a robust package for solving the complex Monge-Amp\`ere equation on non-compact complex manifolds and applied their techniques to construct complete Ricci-flat metrics $\omega_{TY}$ on $X=Y\setminus D$ asymptotic to the Calabi ansatz.   Hein \cite{Hein} later improved and sharpened the Tian-Yau method, obtaining the result that in a Tian-Yau space the pair $(\omega_{TY},\Omega)$ converges to $((\Phi^{-1})^*\omega_{\mathcal{C}},(\Phi^{-1})^*\Omega_{\mathcal{C}})$ with exponential decay for a suitable map $\Phi:X_{\mathcal{C}}\rightarrow X$ which is a diffeomorphism onto is image \cite{HSVZ}. Therefore, we expect to be able to transfer geometric properties of $X_{\mathcal{C}}$ to a neighborhood of infinity in $X$. 
     
   \subsection{Special Lagrangians in Tian-Yau spaces}
    
    We now explain how to produce special Lagrangians in Tian-Yau spaces.   The notion of a special Lagrangian was introduced by Harvey-Lawson \cite{HL} in their foundational work on calibrated geometries. Let $X$ be a Calabi-Yau manifold with holomorphic volume form $\Omega$ and Ricci-flat metric $\omega$. A half dimensional submanifold $L$ is a {\em Lagrangian} if $\omega|_L=0$. From a local coordinate calculation, one has $\Omega|_L=e^{i\theta}\mbox{vol}_L$, where $\theta:L\rightarrow S^1$ is called the {\em phase function} of $L$.  A Lagrangian is {\em graded} if $\theta:L\rightarrow S^1$ can be lifted to a map $\theta: L\rightarrow \mathbb{R}$.  $L$ is said to be special Lagrangian if $\theta$ is a constant. As calibrated submanifolds, special Lagrangians are volume minimizing in their homology class and they play a central role geometry and physics. Typical examples of of special Lagrangians are obtained from
    \begin{itemize}
    \item  the real locus of anti-holomorphic anti-symplectic involutions,
    \item  hyperK\"ahler rotation of holomorphic Lagrangians in hyperK\"ahler manifolds \cite{HL},
    \item  deformations of known special Lagrangians \cite{McLean},
    \item  lifting of special Lagrangians in symplectic reductions \cite{G00}.
    \end{itemize}

     Let us now return to the setting of the Calabi ansatz.   Suppose we are given any special Lagrangian $L\subseteq D$ with respect to $\omega_D,\Omega_D$ of phase $\theta$, i.e. $\Omega_D|_L=e^{i\theta}\mbox{vol}_L$, we may associate a submanifold 
      \begin{align*}
        L_{\epsilon}=p^{-1}(L)\cap \{|\xi|_h=\epsilon\}\subseteq X,
      \end{align*} 
      which is again a special Lagrangian with respect to $\omega_{\mathcal{C}},\Omega_{\mathcal{C}}$ of phase $\theta+\frac{\pi}{2}$ by direct calculation. In particular, when $D$ is an elliptic curve or $K3$ surface with a special Lagrangian fibration, \footnote{For instance, any algebraic $K3$ surface admits a special Lagrangian fibration \cite{CJL}} then $X_{\mathcal{C}}$ admits a special Lagrangian fibration as well. 
      
      Using the fact that the Tian-Yau space is asymptotic to the Calabi ansatz \cite{HSVZ}, one may run Moser's trick on $\Phi(L_{\epsilon})$ to produce a Lagrangian submanifold $M_{\epsilon}$ of $X$ near infinity whose phase function is of order $O(e^{-C/\epsilon})$ for some positive constant $C$. In particular, $M_{\epsilon}$ is a graded Lagrangian.   The idea is to perturb $M_{\epsilon}$ to a genuine special Lagrangian in the Tian-Yau space.  One approach to this perturbation problem is to use the Lagrangian mean curvature flow to deform $M_{\epsilon}$, as we will explain below.
      
  In general, if $M$ is a graded Lagrangian in a Calabi-Yau manifold then the mean curvature is given by $\vec{H}=J\nabla \theta$, where $J$ is the almost complex structure, and $\nabla \theta$ is the Riemannian gradient of the phase function. The mean curvature flow (MCF) produces a time-dependent family of immersions $F_t: M \subseteq X$ satisfying the equation 
         \begin{align*}
             \frac{\partial }{\partial t}F_t =\vec{H}. 
          \end{align*}
     Smoczyk showed that the Lagrangian condition is preserved by the MCF in K\"ahler-Einstein manifolds \cite{Smo}, thus in this case it is called Lagrangian mean curvature flow (LMCF). If one can prove the smooth convergence of the flow, then the limit of LMCF has $\nabla \theta=0$ and is thus a special Lagrangian. However, examples of Neves show that the LMCF of an arbitrary small $C^0$ deformation of a special Lagrangian can develop finite time singularities \cite{Ne}. Worse still, in the setting of the Tian-Yau space the geometry is not bounded due to the degeneration of the injectivity radius, and hence many of the standard tools in the literature do not apply. Nevertheless, the authors obtained the following theorem:    
    \begin{thm}\cite[Theorem 1.1]{CJL} \label{LMCF conv}
    	For $\epsilon\ll 1$, the LMCF of $M_{\epsilon}$ converges smoothly to a smooth special Lagrangian submanifold. 
    \end{thm}
     In particular, this produces many new examples of compact special Lagrangian submanifolds in complete Calabi-Yau spaces.
    
    % \begin{que}
    % 	The Tian-Yau metric can also be constructed in the case when $Y$ admits orbifold singularities. Can Theorem \ref{LMCF conv} be generalized to this case?  {\red I'm a little afraid this might be too trivial to make a question on its own}.
    % \end{que}
    \subsection{Special Lagrangian fibrations in del Pezzo surfaces}
      In this section we restrict to the case $Y$ is a weak del Pezzo surface and explain the existence of special Lagrangian fibrations in $X$. From the classification of compact complex surfaces, a weak del Pezzo surface is either $\mathbb{P}^1\times \mathbb{P}^1$, the Hirzebruch surface $\mathbb{F}_2$ or the successive blow-up of $\mathbb{P}^2$ at $d\leq 8$ points with no point lying on a $(-2)$-curve.
      
      \subsubsection{Local fibration near infinity}
      When $Y$ is a weak del Pezzo surface, a smooth divisor $D \in |-K_{Y}|$ is necessarily an elliptic curve and hence admits a special Lagrangian fibration with fibre class $\gamma$ for any primitive $\gamma\in H_1(D,\mathbb{Z})$. From the discussion in the previous section, such a special Lagrangian fibration lifts to a special Lagrangian fibration in $X_{\mathcal{C}}$ and produces an almost special Lagrangian fibration near infinity in $X$. Moreover, by Theorem~\ref{LMCF conv} the LMCF flows every such fibre to a special Lagrangian torus.  A priori, it is not clear if the LMCF preserves the fibration structure-- for example, by considering a dumbbell, it is easy to see even the curve shortening flow can evolve disjoint, almost minimal submanifolds towards the same minimal limit. However, in our case we can exploit the symplectic structure; take two disjoint almost special Lagrangians $M_{\epsilon},M'_{\epsilon'}$ in $X$ and assume that they converge under the LMCF to special Lagrangians $\tilde{M}_{\epsilon}=\tilde{M}'_{\epsilon'}$.   Since LMCF preserves the Hamiltonian isotopy class of Lagrangians and the Floer cohomologies between Lagrangians are independent of the representatives in the Hamiltonian isotopy classes, we have 
        \begin{align*}
           0=HF(M_{\epsilon},M'_{\epsilon'})\cong HF(\tilde{M}_{\epsilon},\tilde{M}_{\epsilon'})\cong HF(\tilde{M}_{\epsilon})\cong H^*(T^2)\neq 0
        \end{align*} 
        which is a contradiction. \footnote{Here the ambient geometry is not of bounded geometry, so to define the Floer cohomology one needs a version of Sachs-Uhlenbeck-Gromov compactness theorem in this setting which is provided in \cite[Propostion 5.3]{CJL} or \cite{Gro}.} Here we use the fact that in a hyperK\"ahler surface Maslov index zero Lagrangians bound no $J$-holomorphic discs with respect to a generic almost complex structure for dimensional reasons and thus the Floer cohomology reduces to the usual cohomology.  This shows that $\tilde{M}_{\epsilon}\ne \tilde{M}'_{\epsilon'}$. On the other hand, after a suitable hyperK\"ahler rotation (as explained below), $\tilde{M}_{\epsilon}, \tilde{M}'_{\epsilon'}$ become holomorphic. In particular, they can have only finitely many intersections and each intersection contributes positively.  Since their topological intersection is zero, this shows that in fact $\tilde{M}_{\epsilon}, \tilde{M}'_{\epsilon'}$ are disjoint. Therefore, the LMCF flows the almost special Lagrangian fibration to a genuine special Lagrangian fibration near infinity. 
        
        \subsubsection{HyperK\"ahler rotation and compactification}
        Next we will explain the hyperK\"ahler rotation trick. Since $Sp(1)\cong SU(2)$, every Calabi-Yau surface is hyperK\"ahler. Given a Calabi-Yau surface $X$ with holomorphic volume form $\Omega$ and Ricci-flat metric $\omega$ such that $2\omega^2=\Omega\wedge \bar{\Omega}$, one can associate an $S^2$-family of complex structures and corresponding Ricci-flat K\"ahler forms on the same underlying Riemannian manifold $X$. For simplicity, we will single out an $S^1\subseteq S^2$ with coordinate $\theta$ where the corresponding K\"ahler form and holomorphic volume form are given by 
           \[
              \omega_{\theta}=\mbox{Re}(e^{-i\theta}\Omega), \qquad \Omega_{\theta}= \omega-i\mbox{Im}(e^{-i\theta}\Omega). 
           \]
           We let $X_{\theta}$ denote the space $X$ equipped with the K\"ahler form $\omega_{\theta}$ and holomorphic volume form $\Omega_{\theta}$. By the Wirtinger theorem,  $C\subseteq X$ is a holomorphic curve if and only if $C$ is special Lagrangian in the same underlying space but with the Calabi-Yau structure given by $X_{\theta}$.  In particular, for hyperK\"ahler surfaces there is a close connection between holomorphic cycles and special Lagrangian cycles. 
        
        Now, by Theorem~\ref{LMCF conv} and the above discussion we have constructed a special Lagrangian torus fibration near infinity in $X$, and this becomes a genus-one holomorphic fibration after a suitable hyperK\"ahler rotation.  Let us denote this hyperK\"ahler rotation by $\check{X}$ with its K\"ahler form $\check{\omega}$ and holomorphic volume form $\check{\Omega}$ given by 
           \begin{align}\label{HK rel}
            \check{\omega}=\mbox{Re}\Omega, \hspace{3mm} \check{\Omega}=\omega-i\mbox{Im}\Omega.
           \end{align}
        It is not too hard to see that there exist local sections of the fibration. Thus, there exists a compact set $K\subseteq \check{X}$ such that $\check{X}\setminus K\rightarrow \Delta^*$ is an elliptic fibration, where $\Delta^*$ is the punctured disc. In particular, there exists a holomorphic map from $\Delta^*$ to the moduli space $\mathcal{M}_{1,1}$ of elliptic curves with a marked point. Since the topological monodromy of the fibration around infinity is conjugate to $\begin{pmatrix} 1 & d \\ 0 & 1 \end{pmatrix}$, the elliptic fibration $\check{X}\setminus K\rightarrow \Delta^*$ is the pull-back of the universal family over $\mathcal{M}_{1,1}$ by \cite[Proposition 5.9]{Hain}. Moreover, the $j$-invariants of the elliptic fibration near infinity in $\check{X}$ converge to infinity and the holomorphic map $\Delta^*\rightarrow \mathcal{M}_{1,1}$ extends to $\Delta\rightarrow \overline{\mathcal{M}}_{1,1}$. Pulling back the universal family over $\overline{\mathcal{M}}_{1,1}$ gives a compactification $\check{Y}$ of $\check{X}$. From the monodromy of the fibration around infinity, the added fibre $\check{D}$ at infinity is of type $I_d$ in Kodaira's classification. To sum up, a suitable hyperK\"ahler rotation $\check{X}$ of $X$ can be compactified to a compact complex surface by adding an $I_d$-fibre.
        
      \subsection{Implications of the Enriques-Kodaira Classification}  
        Now we can use the Enriques-Kodaira classification of compact complex surfaces to deduce the structure of $\check{Y}$.  Recall the following definition \cite{HeckLoo}:
        
        \begin{defn}\label{defn: RES}
        A rational elliptic surface $\check{Y}$ is a smooth, compact complex surface satisfying any of the following equivalent definitions:
        \begin{enumerate}
        \item[$(i)$] $\check{Y}$ is  rational surface with an elliptic fibration structure possessing a section.
        \item[$(ii)$] $\check{Y}$ can be realized as the blowup of the base locus of a pencil of cubics in $\mathbb{P}^2$.
        \item[$(iii)$] $-K_{\check{Y}}$ is base-point free and defines an fibration.
        \end{enumerate}
        \end{defn}
        
        Let $\check{\omega}, \check{\Omega}$ be the corresponding Ricci-flat K\"ahler form and holomorphic volume form on $\check{X}$ obtained from the hyperK\"ahler rotation, as above. By a local calculation one checks $\check{\Omega}$ extends to a meromorphic $2$-form on $\check{Y}$ with a simple pole along $\check{D}$, where $\check{D}=\check{Y}\setminus \check{X}$. Therefore, we have $K_{\check{Y}}=\mathcal{O}_{\check{Y}}(-\check{D})$. From the elliptic fibration on $\check{Y}$, we have 
        $c_1(\check{Y})^2=0$.  
        There are no $(-1)$ curves in the fibre by the  adjunction formula. Since del Pezzo surfaces are either a blowup of $\mathbb{P}^2$ or $\mathbb{P}^1\times \mathbb{P}^1$, we have $b_1(Y)=0$. From the Mayer-Vietoris sequences, we have 
        \begin{align*}
        b_1(\check{Y})=b_1(\check{X})=b_1(X)=b_1(Y)=0.
        \end{align*}
         If  $\check{Y}$ is minimal, then $\check{Y}$ can only be an Enriques surface, a $K3$ surface, or a minimal, properly elliptic surface by the Enriques-Kodaira classification (see for example \cite[Chapter VI, Table 10]{BPV}). Since a properly elliptic surface has Kodaira dimension $1$, the fact that $K_{\check{Y}}=\mathcal{O}_{\check{Y}}(-\check{D})$ rules out all three possibilities and hence $\check{Y}$ is not minimal. 
        
        Now, any $(-1)$ curve $E$ in $\check{Y}$ has intersection one with $\check{D}$ and so $(\check{D}+E)^2>0$. Therefore, $\check{Y}$ is projective by \cite[Chapter IV, Theorem 5.2]{BPV}. Then $h^1(\check{Y},\mathcal{O}_{\check{Y}})=0$ from Hodge theory and $h^0(\check{Y},K_{\check{Y}}^2)=0$ since $-K_{\check{Y}}$ is effective. Finally, Castelnuovo's rationality criterion implies that $\check{Y}$ is rational. Then from the Riemann-Roch theorem, we have $H^0(\check{Y},\mathcal{O}_{\check{Y}}(\check{D}))=2$ and the local elliptic fibration near $\check{D}$ in $\check{Y}$ actually extends to an elliptic fibration. Indeed, one has $\mbox{Pic}(\check{Y})\cong H^2(\check{Y},\mathbb{Z})$ since $H^1(\check{Y},\mathcal{O}_{\check{Y}})=H^2(\check{Y},\mathcal{O}_{\check{Y}})=0$. Thus, $\check{Y}$ is a rational elliptic surface. In particular, there exists an elliptic fibration $\check{X}\rightarrow \mathbb{C}$.  Reversing the hyperK\"ahler rotation one obtains a special Lagrangian $X\rightarrow \mathbb{R}^2$.  Alternatively, one may argue directly from the deformation and compactness theory of $J$-holomorphic curves that the fibration on $X\setminus K$ extends to a holomorphic fibration in all of $X$; see \cite{CJL}.  To summarize, we have the following theorem 
        
        \begin{thm}\cite[Theorem 1.3, Theorem 1.6]{CJL} \label{CJL main thm}
        	Let $Y$ be a del Pezzo surface of degree $d$ and $D$ a smooth anti-canonical divisor, then $X=Y\setminus D$ admits a special Lagrangian fibration. Moreover, a suitable hyperK\"ahler rotation $\check{X}$ given by \eqref{HK rel} of $X$ can be compactified to a rational elliptic surface $\check{Y}$ by adding an $I_d$ fibre. 
        \end{thm}
       \begin{rk}
       	  In the case of $Y=\mathbb{P}^2$, this proves a conjecture of Auroux \cite[Conjecture 5.1]{Aur}. It is worth mentioning that the corresponding rational elliptic surface $\check{Y}$ has the singular configuration $I_9I_1^3$. For readers interested in mirror symmetry, it is not hard to connect this to the mirror superpotential of $\mathbb{P}^2$. The mirror superpotential of $\mathbb{P}^2$ is $x+y+\frac{1}{xy}:(\mathbb{C}^*)^2\rightarrow \mathbb{C}$ and it happens to be compactified to the unique extremal rational elliptic surface with singular configuration $I_9I_1^3$. This hints at the relation between mirror symmetry and hyperK\"ahler rotation; see Section \ref{sec: ms} for more discussion.
       \end{rk}
       \begin{rk}
       	  When $Y$ is a rational elliptic surface and $D$ is a smooth elliptic fiber, Tian-Yau \cite{TY} constructed an asymptotically cylindrical Ricci-flat metric \cite[Theorem 5.2]{TY}. The argument in theorem \ref{CJL main thm} also applies in this case, see \cite[Theorem 1.3]{CJL}. In particular, the special Lagrangian fibration in $X$ for a generic choice of $(Y,D)$ admits $12$ singular fibres of type $I_1$. This  confirms another conjecture of Auroux \cite[Conjecture 2.10]{Aur}.
       \end{rk}
       \begin{rk}
       	   Recall that every primitive class  $\gamma\in H_1(D,\mathbb{Z})$, generates a special Lagrangian fibration in $X$ with fibre homology class the $S^1$ product with the special Lagrangian in the homology class of $\gamma$ in $D$. In particular, there are countably many different special Lagrangian fibration structures in $X$. 
       \end{rk}

\section{special Lagrangian Fibrations in Rational Elliptic Surfaces}
 In this section we discuss Calabi-Yau metrics and special Lagrangian fibrations on rational elliptic surfaces. Suppose $\check{Y}\rightarrow \mathbb{P}^1$ is a rational elliptic surface, as described in Definition~\ref{defn: RES}.  From the canonical bundle formula of blowups, any fibre of the elliptic fibration is an anti-canonical divisor. Choose any fibre $\check{D}\subseteq \check{Y}$ and let $\check{X}=\check{Y}\setminus \check{D}$, there exists a meromorphic form $\check{\Omega}$ on $\check{Y}$ with simple pole along $\check{D}$ unique up to $\mathbb{C}^*$-scaling.  Consider the long exact sequence of pairs $(\check{Y},\check{X})$ with integer coefficients; we have
  \begin{align} \label{les}
  0\rightarrow  H^1(\check{D})\rightarrow H_2(\check{X})\rightarrow H_2(\check{Y})\rightarrow H^2(\check{D}).
  \end{align} 
  where we use the Poincar\'e duality $H^k(\check{D})\cong H_{4-k}(\check{Y},\check{X})$. Let $[C]\in H_2(\check{X})$ be the image of the generator of $H^1(\check{D})\cong \mathbb{Z}$, which is called the ``bad cycle" in $\check{X}$. 
   
   \begin{thm}[Hein, \cite{Hein}]~\label{thm: hein}
   For any K\"ahler form $\check{\omega}'$ on $\check{X}$ with $\int_{\check{X}}(\check{\omega}')^2<\infty$ and $[\check{\omega}'].[C]=0$ for $[C]$ being the ``bad cycle" \footnote{One can see that $[C]$ is called the ``bad cycle" as it seems to be an obstruction of the existence of the Ricci-flat metric. }, there exists a complete Ricci-flat metric $\check{\omega}$ on $\check{X}$ such that 
  \begin{enumerate}
  	\item $\check{\omega}^2=\alpha \Omega\wedge \bar{\Omega}$ for some sufficiently large $\alpha$. 
  	\item $\check{\omega}-\check{\omega}'$ is $d$-exact.
  	\item $\check{\omega}$ decays exponentially fast to some {\bf semi-flat metric} near infinity. 
  \end{enumerate}
   	
\end{thm}

 We will now explain the {\bf semi-flat metric} appearing in Theorem~\ref{thm: hein}. For the purpose of this survey, we will only consider the case where $\check{D}$ is an $I_d$-fibre. Consider the model space
    \begin{align*}
      \check{X}_{mod}=  \Delta^*\times \mathbb{C}/\Lambda(u), \mbox{ where } \Lambda(u)=\mathbb{Z}\oplus \mathbb{Z}\frac{d}{2\pi i}\log{u},
    \end{align*} where $u$ is the coordinate on the punctured disc $\Delta^*$. $\check{X}_{mod}$ admits a natural partial compactification $\check{Y}_{mod}$ by adding an $I_d$-fibre over $u=0$. Any germ of an elliptic fibration over $\Delta$ with only one singular fibre of type $I_d$ over $u=0$ is biholomorhpic to $\check{X}_{mod}$ by the Abel-Jacobi map after a choice of a local section near infinity. Consider the holomorphic volume form $\check{\Omega}_{sf}=\frac{\kappa(u)}{u}dv\wedge du$, where $v$ is the coordinate of the elliptic fibres and $\kappa(u)$ is a holomorphic function in $u$.  Following the work of Greene-Shapere-Vafa-Yau \cite{GSVY} we consider the semi-flat metric 
    \begin{align}\label{sf metric}
       \check{\omega}_{sf,\epsilon} &:=  \sqrt{-1}|\kappa(u)|^2 \frac{k|\log|u||}{2\pi\epsilon}\frac{du\wedge d\bar{u}}{|u|^2} \notag \\
       &\quad + \frac{\sqrt{-1}}{2} \frac{2\pi\epsilon}{k|\log|u||} \left(dv+B(u,v)du\right)\wedge \overline{\left(dv+B(u,v)du\right)},
    \end{align} where $B(u,v) = -\frac{{\rm Im}(v)}{\sqrt{-1}u|\log|u||}$ and $\epsilon$ is the volume of the elliptic fibre.  It is easy to see that $\check{\omega}_{sf}$ restricts to a flat metric on each fibre and one can check directly that $2\check{\omega}_{sf}^2=\check{\Omega}_{sf}\wedge \bar{\check{\Omega}}_{sf}$, i.e. $\check{\omega}_{sf}$ is a hyperK\"ahler metric. 
    
    The homology $H_2(\check{X}_{mod},\mathbb{Z})$ is of rank two and is generated by the fibre and the ``bad cycle" $[C]\in H_2(\check{X}_{mod},\mathbb{Z})$, which is the unique homology class such that its push-forward in $H_2(\check{Y}_{mod},\mathbb{Z})$ is homologous to zero. Explicitly, $[C]$ can be realized by the surface $\{|u|=const, \mbox{Im}(x)=const\}$. Straightforward calculation shows that these are special Lagrangian tori with respect to $(\check{\omega}_{sf},\check{\Omega})$ and these cycles clearly define a fibration of $\check{X}_{mod}$.
    
    \begin{rk}
    A semi-flat metric on $\check{X}$ is induced by the semi-flat metric $\omega_{sf, \epsilon}$ on $\check{X}_{mod}$, together with a choice of a holomorphic section $\sigma$ of $\check{X}$ near infinity which enters via the identification of a neighborhood of infinity in $\check{X}$ with $X_{mod}$ . It is essential that different choices of section can yield radically different K\"ahler metrics; e.g. not de Rham cohomologous and not uniformly equivalent.  Thus, a semi-flat metric on $\check{X}$ should be denoted $\omega_{sf, \sigma, \epsilon}$ in order to record the dependence on $\sigma$.  However, for the purposes of this survey, we shall suppress the dependence on $\sigma$.
    \end{rk}
    % $\sigma:\Delta^*\footnote{Although we suppress the choice of local section at infinity, it is essential that different choices of local section result in the same compactification up to biholomorphism, but may result in homologically distinct bad cycles.} 
    
    Thanks to Hein's robust estimates the same argument as in Theorem \ref{LMCF conv} and Theorem \ref{CJL main thm} proves that LMCF of the above special Lagrangian torus model converges to a genuine special Lagrangian with respect to $(\check{\omega},\check{\Omega})$ and can be deformed to a special Lagrangian. To summarize, we have 
     \begin{thm}\label{SLag fibration RES} \cite[Theorem 3.4]{CJL2}
     	Given a rational elliptic surface $\check{X}$ and an $I_d$-fibre $\check{D}$, there exists a special Lagrangian fibration on $\check{X}=\check{Y}\setminus \check{D}$ with respect to any of the Ricci-flat K\"ahler metrics produced by Hein. Moreover, a suitable hyperK\"ahler rotation $\check{X}'$ of $\check{X}$ can be compactified to a rational elliptic surface $\check{Y}'$ by adding an $I_d$-fibre. 
     \end{thm}
      \begin{rk}
    	If $\check{D}$ is a singular fibre of any type other than $I_d$, then by the Mayer-Vietoris sequence the second homology of a neighborhood of infinity of $\check{X}$ is generated by the elliptic fibre up to torsion elements. In particular, $\check{X}$ cannot admit a special Lagrangian fibration in these cases. 
    \end{rk}

\section{Non-standard semi-flat metrics}
    
  It was observed by Hein that the Ricci-flat metric on $\check{X}$ constructed by Theorem~\ref{thm: hein} has the same general geometric quantities, including curvature and injectivity radius decay and volume growth as the $2$-dimensional Tian-Yau spaces; see ~\eqref{eq: geomTY}. Therefore, there are two natural questions to ask: 
\begin{enumerate}
\item[(Q1)] Is $\check{X}$ a suitable hyperK\"ahler rotation of $X$ as a complex manifold? 
\item[(Q2)] If the answer to (Q1) is affirmative, does the Tian-Yau metric agree with Hein's metric after hyperK\"ahler rotation?
  \end{enumerate}
  Question (Q1) is already answered positively in Theorem \ref{CJL main thm}, however the answer to (Q2) turns out to be negative.  Indeed,  assume that the answer to the second question is affirmative. Choose a weak del Pezzo surface $Y$ of degree $d$ and a generic smooth anti-canonical divisor $D$ with $X=Y\setminus D$. There exist a special Lagrangian fibration in $X$ by Theorem \ref{CJL main thm}. Moreover, there is a suitable hyperK\"ahler rotation $\check{X}$ of $X$ which can be compactified to a rational elliptic surface $\check{Y}$ by adding an $I_d$-fibre. If the hyperK\"ahler metric on $\check{X}$ is one of Hein's metric, then $\check{X}$ admits a special Lagrangian fibration from Theorem \ref{SLag fibration RES}. One can do a reverse hyperK\"ahler rotation to obtain a new special Lagrangian fibration in $X$, different from the one we started with, and with phase differing by exactly $\pi/2$. This implies that the fundamental domain of $D$ can be realized as a rectangle, which is absurd as we chose a generic $D$ in the beginning. To sum up, the special Lagrangian fibrations on the weak del Pezzo surfaces and rational elliptic surfaces seem detect some previously unknown hyperK\"ahler metrics on $\check{X}$. In particular, this suggests that there are more hyperK\"ahler metrics on $\check{X}_{mod}$ than those of the form \eqref{sf metric}.  This motivates the search for new ansatzes of Ricci-flat metrics on $\check{X}_{mod}$. It was discovered in \cite[Section 2]{CJL2} that for each $b_0\in \mathbb{R}$, there is a {\it non-standard semi-flat metric} given by 
     \begin{align*}
   \check{\omega}_{sf, b_0, \epsilon} &:= \sqrt{-1}\frac{|\kappa(u)|^2}{\epsilon} W^{-1}\frac{du\wedge d\bar{u}}{|u|^2}\\
   &\quad + \frac{\sqrt{-1}}{2} W \epsilon \left(dv+\widetilde{\Gamma}(v,u,b_0) du\right) \wedge \overline{\left(dv+\widetilde{\Gamma}(u,v,b_0) du\right)},
   \end{align*} where $W= \frac{2\pi}{k|\log|u||}$ and $
   \widetilde{\Gamma}(u,v,b_0) = B(u,v)+ \frac{b_0}{2\pi^2}\frac{|\log|u||}{u}$. 
   \begin{rk}
   	   	If $2b_0/d\in \mathbb{Z}$, then $\omega_{sf,b_0,\epsilon}$ is the standard semi-flat metric after fiberwise translation by a section $\sigma:\Delta^*\rightarrow \check{X}_{mod}$. Otherwise, $\check{\omega}_{sf,b_0,\epsilon}$ provides a new Ricci-flat metric on $\check{X}_{mod}$. Straightforward calculations show that the restriction of $\check{\omega}_{sf,b_0,\epsilon}$ to the elliptic fibres is also flat. In other words, there exists a $1$-parameter family of semi-flat metrics on $\check{X}_{mod}$, parametrized by $b_0$, for which a discrete set of choices give the standard semi-flat metric of Greene-Shapere-Vafa-Yau.
   \end{rk}
    
   Geometrically, the non-standard semi-flat metric can be obtained from the standard semi-flat metric by pulling back along the fibrewise translation map coming from a certain multi-section (possibly uncountably valued). Consider a multi-section $\sigma: \Delta^*\rightarrow \check{X}_{mod}$ of the form 
      \begin{align*}
         \sigma: &\Delta^*\longrightarrow \check{X}_{mod}\\
           & u\mapsto \bigg(u,h(u)+ \frac{c_0}{2\pi i}\log{u}+\frac{b_0}{(2\pi i)^2}(\log{u})^2 \bigg)
      \end{align*} for $b_0,c_0\in \mathbb{R}$ and $h(u)$ a holomorphic function on $\Delta^*$. Let $T^*_{\sigma}$ be the fibrewise translation by the multi-section $\sigma$. Straightforward calculation shows that $T^*_{\sigma}\check{\omega}_{sf}=\check{\omega}_{sf,b_0,sf}$ is well-defined on $\check{X}_{mod}$. It is worth noticing that $2b_0/d\in \mathbb{Z}$ if and only if $\sigma$ extends to a section over $\Delta\rightarrow \check{Y}_{mod}$, \cite[Lemma 3.28]{FM}. 
      
       The non-standard semi-flat metric $\check{\omega}_{sf,b_0,\epsilon}$ can be used as a model metric near $\check{D}$ to construct complete hyperK\"ahler metrics on $\check{X}$ following the arguments of Hein \cite{Hein}. Denote $[F]\in H_2(\check{X},\mathbb{Z})$ the homology class of the fibre. By choosing a local section near $\check{D}$, we may identify the elliptic fibration in $\check{X}$ near infinity with the local model $\check{X}_{mod}$. 
       \begin{thm} \cite[Theorem 2.16]{CJL2} \label{generalized Hein's metric} 
       	  Let $\check{\omega}_0$ be a K\"ahler form on $X$ such that 
       	    \begin{align*}
       	       [\check{\omega}_0]_{dR}\cdot [F]=\epsilon, \hspace{5mm} [\check{\omega}_0]_{dR} \cdot [C]=\frac{2b_0}{\epsilon},
       	    \end{align*} for some $b_0\in \mathbb{R}$. Then there exists a K\"ahler metric $\check{\omega}_{\alpha}$ on $\check{X}$ such that 
       	    \begin{align*}
       	    \check{\omega}_{\alpha}= \alpha T^*_{h}\check{\omega}_{sf,b_0,\frac{\epsilon}{\alpha}}
       	    \end{align*} near $\check{D}$ and $\check{\omega}_{\alpha}=\check{\omega}_0$ outside of a neighborhood of $\check{D}$. Here $h$ is a holomorphic section over $\Delta^*$ which depends only on the Bott-Chern cohomology class $[\check{\omega}_0]_{BC} \in H^{1,1}_{BC}(X, \mathbb{R})$. Moreover, for sufficiently large $\alpha$ \footnote{The condition can be viewed as saying the existence holds sufficiently close to  the large volume limit; see the discussion in Section \ref{sec: ms}}, there exists a hyperK\"ahler metric on $\check{X}$
       	      \begin{align*}
       	         \check{\omega}=\check{\omega}_{\alpha}
       	         +i\partial \bar{\partial} \phi, 
       	      \end{align*} such that 
       	      \begin{enumerate}
       	      \item  $\check{\omega}^2=\alpha^2\check{\Omega}\wedge \bar{\check{\Omega}}$, and
       	      	\item $\check{\omega}$ is asymptotic to the (possibly non-standard) semi-flat metric $T^*_h\check{\omega}$ in the sense that there exists a constant $\delta>0$ such that 
       	      	 \begin{align*}
       	      	    |\nabla^k \phi|\sim O(e^{-\delta r^{2/3}}),
       	      	 \end{align*} for every $k\in \mathbb{N}$, where $r$ is the distance function to a fixed point in $\check{X}$. 
       	      \end{enumerate}
       	    In particular, $\check{\omega}$ and $\check{\omega}_{sf,b_0,\epsilon}$ share the same curvature decay $|\nabla^k Rm|\leq r^{-2-k}$ for each $k\in \mathbb{N}$ and injectivity radius decay $inj\sim r(x)^{-1/3}$. 
       \end{thm}
      \begin{rk}
      	 It is possible to show that for any $\epsilon, b_0$ there exist ambient K\"ahler forms $\check{\omega}_0$ satisfying the assumptions of Theorem~\ref{generalized Hein's metric}; see \cite{CJL2} or the proof of Theorem~\ref{another compactification} below. 
      \end{rk}
      
       For each $b_0$ such that $2b_0/d\in \mathbb{Q}$, there exists a homology class $C_{b_0}\in H_2(\check{X}_{mod},\mathbb{Z})$ such that $\int_{C_{b_0}}\check{\omega}_{sf,b_0,\epsilon}=0$ and there exists a special Lagrangian fibration on $\check{X}_{mod}$ respect to $(\check{\omega}_{sf,b_0,\epsilon},\check{\Omega}_{sf})$ with fibre class $C_{b_0}$. Similar to Theorem \ref{SLag fibration RES}, we have the following theorem 
       \begin{thm} \label{SYZ RES}
       	  With the above notation, if $2b_0/d\in \mathbb{Q}$ then there exists a special Lagrangian fibration on $\check{X}$ with fibre class $C_{b_0}$ with respect to the Ricci-flat metric in Theorem \ref{generalized Hein's metric}. Moreover, a suitable hyperK\"ahler rotation $\check{X}'$ of $\check{X}$ can be compactified to a rational elliptic surface $\check{Y}'$ by adding an $I_d$ fibre. 
       \end{thm}
    
     With the non-standard semi-flat metric, we now can explain the hyperK\"ahler rotation of the Calabi ansatz.  
       \begin{thm}\cite[Appendix A]{CJL2} \label{local model} Assume that $D\cong \mathbb{C}/\mathbb{Z}\oplus \mathbb{Z}\tau$ is an elliptic curve, with $\tau$ in the upper half-plane. With the notations in Section \ref{sec: Calabi ansatz}, consider the hyperK\"ahler rotation of $X_{\mathcal{C}}$ with K\"ahler form  $\check{\omega}_{\mathcal{C}}$ and holomorphic volume form $\check{\Omega}_{\mathcal{C}}$ such that the ansatz special Lagrangian corresponding to $1\in \mathbb{Z}\oplus \mathbb{Z}\tau$ is of phase zero. Then with suitable choice of coordinate one has 
      	\begin{align*}
      	\check{\omega}_{\mathcal{C}}=\alpha\check{\omega}_{sf,b_0,\epsilon}, \hspace{5mm} \check{\Omega}_{\mathcal{C}}=\alpha\check{\Omega}_{sf},
      	\end{align*} 
      	where $b_0=-\frac{1}{2}\mbox{Re}(\tau)d$, $\epsilon=\frac{2\sqrt{2}\pi }{\mbox{Im}(\tau)}$ and $\alpha=\sqrt{d\pi \mbox{Im}(\tau)}$. In particular, there exists a bijection between $\tau \leftrightarrow (b_0,\epsilon)$, i.e., every (possibly non-standard) semi-flat metric can be realized as some hyperK\"ahler rotation of certain Calabi ansatz up to a scaling. 
      \end{thm}

\section{A uniqueness theorem for hyperK\"ahler metrics}
%{\red made it to here}
  A priori, Hein's construction \cite{Hein} or its generalization in Theorem \ref{generalized Hein's metric} may produce many hyperK\"ahler metrics on $\check{X}$ in the same de Rham, or even Bott-Chern cohomology class.  On the other hand, from the original proof of the Calabi conjecture \cite{Y}, one expects that the complex Monge-Amp\`ere equation has a unique solution with suitable prescribed asymptotic behavior. Even with the robust estimates of Hein giving the fast decay of the hyperK\"ahler metrics to the model metric, establishing a uniqueness theorem is not an easy task. The main obstacles are that $\check{X}$ is not Stein, and admits many holomorphic vector fields.  Indeed, since $\check{X}$ is not Stein, two given hyperK\"ahler metrics in the same de Rham cohomology classes may not represent the same Bott-Chern cohomology class. In this section we will sketch a uniqueness theorem for hyperK\"ahler metrics established in \cite{CJL2}.  Roughly speaking, the failure of the $\ddb$-lemma on $\check{X}$ can be related directly to the holomorphic vector fields on $\check{X}$.
  
   \begin{thm} \cite[Theorem 1.4]{CJL2} \label{uniqueness thm} 
   	   Let $\check{\omega}_1,\check{\omega}_2$ be two hyperK\"ahler metric on $\check{X}$ satisfying the following: 
   	    \begin{enumerate}
   	    	\item $[\check{\omega}_1]_{dR}=[\check{\omega}_2]_{dR}\in H^2(\check{X},\mathbb{R})$, 
   	    	\item $\check{\omega}_i^2=\alpha^2\check{\Omega}\wedge \bar{\check{\Omega}}$, for some $\alpha>0$ and $i=1,2$. 
             \item there exists a possibly non-standard semi-flat metric $\check{\omega}_{sf,b_0,\epsilon}$ such that after a suitable choice of local section near infinity, one has 
              \begin{align} \label{poly decay}
                 |\check{\omega}_i-\alpha\check{\omega}_{sf,\sigma, b_0,\epsilon}|_{\check{\omega}_{sf, b_0,\epsilon}}\leq C r^{-\frac{4}{3}}.
              \end{align} 
   	    \end{enumerate} Then $\check{\omega}_1=\check{\omega}_2$. 
   \end{thm} 
%\textcolor{red}{sketch of proof} 
\begin{proof}[Sketch of proof]
Suppose that we are given $\omega_1, \omega_2$ as above.  Let $x_0 \in X$ be a fixed point, and let $B_{\omega_i}(R)$ denote the ball of radius $R$ around $x_0$.  Since $\omega_i$ converge toward a semi-flat metric we have
\[
{\rm Vol}(B_{\omega_i}(R)) \sim R^{4/3}.
\]
Let us assume that we can write $\omega_2 = \omega_1+ \ddb \phi$ where $\phi$ is bounded.  By adding a constant, we may assume that $\phi \geq 0$.  Then, from the Monge-Amp\`ere equation and the AM-GM inequality we have
\[
\Delta_{\omega_1} \phi \geq 0.
\]
Let $r(x)={\rm dist}_{\omega_1}(x_0,x)$, and let $\eta \geq 0$ be a smooth function such that $\eta \equiv 1$ on $[0,1]$ and $\eta \equiv 0$ on $[2,\infty]$ and $|\eta'|\leq 2$.  Define 
\[
\eta_{R} = \eta\left(\frac{r}{R}\right).
\]
Then, for $R\gg 1$ we have
\[
\begin{aligned}
\int_{X} \eta_{R}^2 \Delta_{\omega_1}\phi^2 \omega_1^2 &= 2\int_{X} \eta_{R}^2(|\nabla \phi|^2_{\omega_1} + \phi \Delta_{\omega_1}\phi) \omega_1^2\\
&\geq  2\int_{X} \eta_{R}^2|\nabla \phi|^2_{\omega_1}  \omega_1^2.
\end{aligned}
\]
On the other hand, integration by parts yields
\[
\begin{aligned}
\int_{X} \eta_{R}^2 \Delta_{\omega_1}\phi^2 \omega_1^2 &= -4 \int_{X} \eta_{R} \phi \langle \nabla \eta_{R}, \nabla \phi \rangle_{\omega_1} \omega_1^2\\
& \leq \int_{X}\eta_{R}^2|\nabla \phi|^2\omega_1^2 + C\int_{X}\phi^2 |\nabla\eta_{R}|_{\omega_1}^2 \omega_1^2.
\end{aligned}
\]
Combining these inequalities, and using that $\phi$ is bounded and $|\nabla \eta_{R}| \leq 2R^{-1}$ yields
\[
\int_{X} \eta_{R}^2|\nabla \phi|^2_{\omega_1}  \omega_1^2 \leq C R^{-2}\int_{B(2R)}\omega_1^2 \leq CR^{4/3-2}
\]
Taking the limit as $R\rightarrow +\infty$ yields $\nabla \phi=0$, and hence $\phi$ is constant.  Thus, we have reduced the uniqueness theorem to producing a bounded potential function $\phi$.  

To produce a bounded potential we argue as follows.  Write $\omega_2=\omega_1+d\beta$. Let $\omega_{sf,b_0, \epsilon}$ is induced by some section $\sigma$ near $\infty$.  We can find a rational elliptic surface $Y$ compactifying $X$ such that $\sigma$ extends as a holomorphic section over the $I_d$ singular fiber $Y\setminus X$.  By assumption we have
\begin{equation}\label{eq: decaydB}
|d\beta|_{\omega_{sf, b_0, \frac{\epsilon}{\alpha}}} \leq Cr^{-4/3}.
\end{equation}
By comparing the semi-flat metric with a smooth K\"ahler metric in coordinates near the $I_d$ fiber, one observes that $d\beta$ has $L^1$ coefficients, and hence extends as a $(1,1)$ current $T$ on $Y$.  One then shows that $T$ is closed, and that $[T]_{dR} =0$ in $H^{2}(Y,\mathbb{R})$.  Thus, by the $\ddb$-lemma for currents we can write $T=\ddb \phi$ for an $L^1$ function $\phi$.  We now need to show that $\phi$ is bounded.  We may as well work in $\check{X}_{mod}$ with coordinates $(u,v)$ as before. The bound~\eqref{eq: decaydB} with $d\beta = \ddb \phi$ implies that, along flat the torus fibers of $\check{X}_{mod}\rightarrow \Delta^*$ we have $|\phi_{v\bar{v}}| \leq C\frac{1}{(-\log|u|)^3}$ and so $\phi$ is roughly harmonic along the fibers.  Intuitively, this implies the $\phi$ is roughly constant along the fibers.  Indeed, since the fibers are flat tori, an argument using elliptic regularity and the Sobolev and Poincar\'e inequalities yields that
\[
\sup_{\pi^{-1}(u)}\bigg|\phi - \frac{2\pi}{d(-\log|u|)} \int_{\pi^{-1}(u)} \phi \frac{\sqrt{-1}dv\wedge d\bar{v}}{2}\bigg| \leq \frac{C}{(-\log|u|)^{\frac{1}{2}}}
\]
Roughly speaking, this implies that we can assume that $\phi = \phi(u)$ is a function pulled back from the base of the elliptic fibration.  In this case the estimate ~\eqref{eq: decaydB} implies
\[
|\phi_{u\bar{u}}| \leq \frac{C}{|u|^2(-\log|u|)^3}
\]
which easily implies that $u$ can be extended over the $I_d$ fiber as a bounded function.  
\end{proof}

For later applications, the key point is that we only assume the hyperK\"ahler metric $\check{\omega}_i$ has polynomial decay to the semi-flat metric instead of the exponential decay.  Indeed, given a possibly non-standard semi-flat metric, one can study the effect of translation by a section on its asymptotics and its Bott-Chern cohomology class.  With this one can deduce from Theorem \ref{uniqueness thm} the following stronger result: 
  \begin{thm} \cite[Proposition 4.8]{CJL2} \label{uniqueness modulo aut}
  	 Let $\mbox{Aut}_0(\check{X})$ denote the fibre preserving biholomorphisms of $\check{X}$ which are homotopic to the identity. Suppose that $\check{\omega}_1,\check{\omega}_2$ are two complete Ricci-flat metrics on $\check{X}$ such that 
  	 \begin{enumerate}
  	 	\item $\check{\omega}_1^2=\check{\omega}_2^2$,
  	 	\item $[\check{\omega}_1]_{dR}=[\check{\omega}_2]_{dR}$,
  	 	\item there are (possibly non-standard) semi-flat metrics $\omega_{sf, \sigma_i, b_{0,i}, \epsilon_i}$ such that 
		\[
		|\check{\omega}_i-\alpha\check{\omega}_{sf, b_{0,i},\frac{\epsilon_i}{\alpha}}|_{\check{\omega}_i}\leq C r_i^{-\frac{4}{3}}.
		\]
		where $r_i$ is the distance from a fixed point with respect to $\check{\omega}_i$.
		%Both are asymptotically to some (a priori can be different) semi-flat metrics in the same of \eqref{poly decay}.
  	 \end{enumerate} 
	 Then there exists $\Phi\in \mbox{Aut}_0(\check{X})$ such that $\Phi^*\check{\omega}_1=\check{\omega}_2$. %{\red Check this statement}
  \end{thm} 
  
To prove Theorem~\ref{uniqueness modulo aut} the idea is to produce a global section $\tau:\mathbb{C}\rightarrow \check{X}$ such that the map $\Phi$ corresponding to fiberwise translation by $\tau$ is isotopic to the identity and
  \begin{equation}\label{eq: autDecay}
  |\Phi^*\check{\omega}_{sf, b_{0,2},\frac{\epsilon_i}{\alpha}}-\check{\omega}_{sf, b_{0,1},\frac{\epsilon_i}{\alpha}}|_{\check{\omega}_{sf, b_{0,1},\frac{\epsilon_i}{\alpha}}}\leq C r^{-\frac{4}{3}}.
  \end{equation}
 One can then apply Theorem~\ref{uniqueness thm} to conclude.  The importance of requiring only polynomial decay, rather than exponential decay, is due to the fact that there may be no global automorphism $\Phi$ producing exponential decay in~\eqref{eq: autDecay}.
  
  The above theorem motivates the following definition of K\"ahler moduli for $\check{X}$, which sets up the foundation for the SYZ mirror symmetry discussion in Section \ref{sec: ms}. 
  Fix the holomorphic volume form $\check{\Omega}$ and define the moduli of complete, asymptotically semi-flat, hyperK\"ahler metrics of $\check{X}$ as 
  	  \begin{align*}
  	      \check{\mathcal{K}}_{CY}:=\{ \mbox{K\"ahler metric } \check{\omega}| 2\check{\omega}^2=\check{\Omega}\wedge \bar{\check{\Omega}} \mbox{ with \eqref{poly decay}} \}   /\mbox{Aut}_0(\check{X}).
  	  \end{align*}
By Theorem \ref{uniqueness modulo aut}, we can identify $\check{\mathcal{K}}_{CY}$ as a subset of $H^2(\check{X},\mathbb{R})$. It is then proved that $\check{\mathcal{K}}_{CY}$ is an open subset of $H^2(\check{X},\mathbb{R})$. In particular, $\check{\mathcal{K}}_{CY}$ is of dimension $11-d$, if $\check{D}$ is an $I_d$-fibre. Note that the moduli space of pairs $(Y,D)$ consisting of a degree $d$ del Pezzo and a smooth anti-canonical divisor has dimension $10-d$.  In particular, $ \check{\mathcal{K}}_{CY}$ seems to be one dimension too large. This mismatch is corrected by restricting to the subset of K\"ahler classes in $H^2(\check{X},\mathbb{R})$ admitting special Lagrangian fibrations. %We will restrict to a hyperplane in $\check{K}_{CY}$ in the next section to talk about mirror symmetry so that there always exists a special Lagrangian fibration with respect to the Ricci-flat metrics in the pool. 

\section{Applications}
\subsection{SYZ Mirror Symmetry of Log Calabi-Yau Surfaces} \label{sec: ms}
It is known that the mirror of a del Pezzo surface $Y$ is a Landau-Ginzburg model consisting of a holomorphic function $W:(\mathbb{C}^*)^2 \rightarrow \mathbb{C}$ called the superpotential. The superpotential $W$ is mirror to $Y$ in various senses: 
\begin{itemize}
\item on the level of enumerative geometry, the quantum periods of $W$ computes the descendant Gromov-Witten invariants of $Y$ \cite{Giv}\cite{F}\cite{CGGK}.
\smallskip
\item on the level of homological mirror symmetry, the Fukaya category of $X$ is conjectured to be equivalent to the derived category of singularities of $W$ or matrix factorization category of $W$ \cite{FU}\cite{LU} \cite{CHL}. Conversely, the Fukaya-Siedel category of $W$ is equivalent to the derived category of coherent sheaves on $Y$ \cite{AKO}. In particular, Auroux-Kartzarkov-Orlov  \cite{AKO} observed that $W$ can be topologically compactified to a rational elliptic surface by adding an $I_d$ fibre if $Y$ is of degree $d$. 
\smallskip
\item Lunts-Przyjalkowski \cite{LP2} showed that the Hodge numbers in the sense of Kartzarkov-Kontsevich-Pantev \cite{KKP} of both sides coincide. 
\smallskip
\item Doran-Thompson \cite{DT} proposed lattice polarized mirror symmetry between the pairs of del Pezzo surfaces of degree $d$ with a smooth anti-canonical divisor and rational elliptic surfaces with an $I_d$-fibre.
\end{itemize}
 
By the classification of compact complex surfaces, del Pezzo surfaces are either blowups of $\mathbb{P}^2$ at $9-d$ points, $0\leq d\leq 8$, or $\mathbb{P}^1\times \mathbb{P}^1$. Thus, there are in total ten deformation families of pairs consisting of a del Pezzo surface with a smooth anti-canonical divisor. On the other hand, there are also ten deformation families of rational elliptic surfaces with an $I_d$ fibre: one for each $d$, $1\leq d\leq 9$ and one additional family for $d=8$. It is thus not hard to match each deformation families of del Pezzo surfaces of degree $d$ with the deformation families of rational elliptic surfaces with an $I_d$ fibre when $d\neq 8$. One can further identify the two families for $d=8$ topologically. Indeed, the complement of a smooth anti-canonical divisor in $\mathbb{P}^1\times \mathbb{P}^1$ (or the Hirzebruch surface $\mathbb{F}_1$) is diffeomorphic to the complement of an $I_8$ fibre in the rational elliptic surface deformation equivalent to the compactification of the superpotential of $\mathbb{P}^1\times \mathbb{P}^1$ (or $\mathbb{F}_1$ respectively) \cite[p.54]{CJL2}.  
 
Given a special Lagrangian fibration $X\rightarrow B$, there are two integral affine structures on the complement of the discriminant locus \cite{Hit} induced by the symplectic and complex structures respectively. In the case of Calabi-Yau surfaces, the two integral affine structures can be defined as follows: fix a reference point $u_0\in B_0$, where $B_0$ is the complement of the discriminant locus in the base $B$. For each $\gamma\in H_1(L_{u_0},\mathbb{Z})$ and $u\in B_0$ near $u_0$, we choose a path $\phi$ connecting $u_0$ and $u$ and define $C_{\gamma,\phi}$ be a cylinder which is an $S^1$-bundle over the image of $\phi$ with fibres homotopic to $\gamma$ via the parallel transport along $\phi$. Then the symplectic affine coordinate is defined to be 
    \begin{align}\label{symplectic affine}
       a_{\gamma}(u):=\int_{C_{\gamma,\phi}}\omega,
    \end{align} which is well-defined thanks to the special Lagrangian condition. Straightforward calculation shows that the transition functions between different choices of reference point $u_0$ and path $\phi$ fall in $\mbox{GL}(2,\mathbb{Z})\rtimes \mathbb{R}^2$. 
     If one replaces $\omega$ in \eqref{symplectic affine} by $\mbox{Im}\Omega$, then one gets the complex affine coordinates.  These coordinates play a key role in the SYZ conjecture, as reformulated and refined by Gross \cite[Conjecture 6.6]{G01}. To explain SYZ mirror symmetry in more detail, we will have to explain the relevant notions of complex/K\"ahler moduli. Since mirror symmetry is expected to happen near the ``large complex structure limit," we need to define a notion of large complex structure limit in this context. 
 
 Given a del Pezzo surface $Y$ of degree $d$ with a smooth anti-canonical divisor $D$, we fix the exact Tian-Yau metric $\omega_{TY}$ on $X=Y\setminus D$. This corresponds to a choice of monotone symplectic structure on $Y$. Therefore, the mirror rational elliptic should admit a distinguished complex structure. To describe such a rational elliptic surface we first recall the Torelli theorem for Looijenga pairs due to Gross-Hacking-Keel \cite{GHK}, as it applies to our current setting. Let $\check{Y}$ be a rational elliptic surface with an $I_d$ fibre $\check{D}$ and let $\check{X}=\check{Y}\setminus \check{D}$. Then the periods of the pair $(\check{Y},\check{D})$ are given by 
   \begin{align}
      \phi_{(\check{Y},\check{D})}: H_2(\check{X})/\mbox{Im}H^1(\check{D})\rightarrow  \mathbb{C}^*  \notag\\
        \gamma \mapsto \exp{2\pi i\big(\frac{\int_{\gamma}\check{\Omega}}{\int_{[C]}\check{\Omega}} \big)}.
   \end{align} 
   From \eqref{les}, we have $H_2(\check{X})/\mbox{Im}H^1(\check{D})\cong \check{D}^{\perp}$. Here $\check{D}^{\perp}$ is the subgroup of $\mbox{Pic}(\check{Y})$ with zero pairing with each component of $\check{D}$. Then we have $\phi_{(\check{Y},\check{D})}\in \mbox{Hom}(\check{D}^{\perp},\mathbb{C}^*)$.
      From a residue calculation in \cite[Proposition 3.12]{Fr}, this coincides with the restriction of line bundles in $\check{D}^{\perp}$ to $\check{D}$ via the identification $\mbox{Pic}^0(\check{D})\cong \mathbb{C}^*$. 
   With the above understanding, we can describe the moduli space of pairs $(\check{Y},\check{D})$:
  \begin{thm}[Gross-Hacking-Keel, \cite{GHK}]\label{Torelli Looijenga}
  	Assume that $(\check{Y}',\check{D}')$ is in the same deformation family of the pair $(\check{Y},\check{D})$. We identify 
  	$\check{D}^{\perp}\cong \check{D}'^{\perp}$ via parallel transport. Then $\phi_{(\check{Y}',\check{D}')}=\phi_{(\check{Y},\check{D})}$ implies an isomorphism of pairs $(\check{Y}',\check{D})\cong (\check{Y},\check{D})$. Moreover, given any homomorphism $\phi \in  \mbox{Hom}(\check{D}^{\perp},\mathbb{C}^*)$ there exists a pair $(\check{Y}',\check{D}')$ in the same deformation family with period $\phi$. 
  \end{thm} 
  
  This result implies that there is a distinguished pair $(\check{Y}_e,\check{D}_e)$ in the deformation family of rational elliptic surfaces with an $I_d$ fibre such that the periods are trivial. We write $\check{X}_e=\check{Y}_e\setminus \check{D}_e$. We claim that the distinguished complex structure is mirror to the choice of exact Tian-Yau metric. 
  
  Similar to the case of $(\check{Y},\check{D})$, the period of $(Y,D)$ determines the pair in its deformation family by  \cite[Theorem 6.4]{McM} and a similar argument of \cite[Proposition 3.12]{Fr}. Let ${\mathcal{M}}_{cpx}$ \footnote{Technically, we will need to consider the moduli space of marked pairs together with the choice of SYZ fibre class and a volume form such that the SYZ fibre is of phase zero.}denote the moduli space of marked pairs of del Pezzo surfaces and smooth anti-canonical divisors. There is a natural holomorphic map ${\mathcal{M}}_{cpx}\rightarrow \mathbb{C}$ given by taking the $j$-invariant of the smooth anti-canonical divisor, determined by the periods of $\mbox{Im}\big(H^1({D}_0)\rightarrow H_2({X}_0)\big)\cong \mathbb{Z}^2$. Here $({Y}_0,{D}_0)$ is a reference pair and ${X}_0={Y}_0\setminus {D}_0$. The fibre is then biholomorphic to $\mbox{Hom}({D}_0^{\perp},{D}_0)$, determined by the rest of the periods in $H_2({X}_0)$. Notice that here $D_0$ is a smooth elliptic curve and thus we have a natural isomorphism $\mbox{Pic}^0(D_0)\cong D_0$.
  
  For SYZ mirror symmetry we will choose a primitive $2$-cycle $[\check{L}]\in H^2(\check{X}_e,\mathbb{Z})$ which can be represented by a $2$-cycle contained in a neighborhood of $\check{D}$. We define the complexified K\"ahler moduli to be
      \begin{align*}
      \check{\mathcal{M}}_{Kah}=\{\mathbb{B}+i[\omega]\in H^2(\check{X}_e,\mathbb{C})|[\check{\omega}].[\check{L}]=0, [\check{\omega}]\in \mathcal{K}_{CY}\}/ H^2(\check{X}_e,\mathbb{Z}).
      \end{align*}
      \begin{rk}
      	The complexified K\"ahler moduli $\check{\mathcal{M}}_{Kah}$ from different choices of $[\check{L}]$ are naturally identified and mirror to the monodromy action on $\mathcal{M}_{cpx}$. 
    \end{rk}

  We will still need a notion of ``large complex structure limit" for $\check{\mathcal{M}}_{cpx}$, where mirror symmetry is expected to happen. Let $(Y_i,D_i)$ be a sequence of del Pezzo surfaces $Y_i$ with anti-canonical divisors $D_i$ such that 
   \begin{enumerate}
   	\item $Y_i$ converges to a smooth del Pezzo surface $Y_{\infty}$ and
   	\item $D_i$ converges to an irreducible nodal anti-canonical divisor $D_{\infty}$ in $Y_{\infty}$. 
   \end{enumerate}
   In \cite{CJL2} we observed that the special Lagrangian fibration of $Y_i\setminus D_i$ corresponding to the vanishing cycle of $D_i$ 
   is collapsing near infinity. This is exactly the feature of the large complex structure limit from the SYZ point of view \cite{KS} and we will say such a sequence of pairs converges to the large complex structure limit. 
   \begin{que}
   	  Does the SYZ fibration of $X_i$ collapse globally as $i\rightarrow \infty$? Notice that the definition of large complex structure limit here is not the usual deepest boundary stratum of the moduli. 
   \end{que}
   
  With the above geometric set-up, we can state precisely the notion of SYZ mirror symmetry between del Pezzo surfaces and rational elliptic surfaces. %{\red should this only be defined near the LCSL?}
   \begin{thm}[C.-Jacob-L., \cite{CJL2}]\label{SYZMS}
   	  Sufficiently close to a large complex structure limit, there exists a mirror map from $\Psi:\mathcal{M}_{cpx}\rightarrow \check{\mathcal{M}}_{Kah}$ such that the special Lagrangian fibrations on $X_q=Y_q\setminus D_q$ and $\check{X}_{\Psi(q)}=\check{Y}_{\Psi(q)}\setminus \check{D}_{\Psi(q)}$ are dual fibration in the following sense: 
   	  \begin{enumerate}
   	  	\item Let $B_{q},B_{\Psi(q)}$ be the base of the SYZ fibration in $X_q, \check{X}_{\Psi(q)}$. There exists a diffeomorphism $\underline{\Psi}_q: B_{q}\rightarrow B_{\Psi(q)}$ such that $\underline{\Psi}$ exchanges the complex and symplectic affine structures of the two SYZ fibrations. 
   	  	\item The sizes of the two fibres are inverses of each other.
   	  \end{enumerate}
   \end{thm}
   The mirror map $\Psi$ is chosen such that hyperK\"ahler rotation of $X_q$ in Theorem \ref{CJL main thm} and that of $\check{X}_{\Psi(q)}$ in Theorem \ref{SYZ RES} are biholomorphic but the holomorphic volume forms differ by a phase $\pi/2$ and are scaled to take care of the fibre sizes relation. The biholomorphism can be achieved by comparing the periods of the rational elliptic surfaces relative to the $I_d$-fibres and applying Theorem \ref{Torelli Looijenga}. To sum up, the mirror pairs are connected by two hyperK\"ahler rotations which match the complex structures but not the K\"ahler structures. 

There are some questions below that the authors will explore in the future. The first question concerns the compatibility between the SYZ mirror symmetry of the non-compact Calabi-Yau surfaces and the mirror superpotential which is the weighted count of the Maslov index two discs. Notice that the latter depends on the compactification.
\begin{que}
	Given a (complexified) K\"ahler form $\check{\omega}_{\Psi(q)}$ on $\check{X}_e$, does the mirror superpotential recover the del Pezzo surface pair $(Y_q,D_q)$? 
\end{que}
The SYZ mirror symmetry in Theorem \ref{SYZMS} is equivalent to turning off the K\"ahler moduli of the del Pezzo surfaces and complex moduli of rational surface pairs. The ultimate SYZ conjecture should include both the A-model and the B-model simultaneously. 
\begin{que}
	Since there are non-exact Tian-Yau metrics, can the SYZ mirror symmetry in Theorem \ref{SYZMS} be extended to include the K\"ahler moduli of del Pezzo surfaces and complex moduli of rational elliptic surfaces?
\end{que}	
One can also ask about the compatibility of different aspects of mirror symmetry in the reverse direction. 
\begin{que}
	If the answer to the above question is affirmative, do the SYZ mirrors coincide with the known superpotentials for non-montone symplectic forms of del Pezzo surfaces? 
\end{que}
 	
 At the end of this section, we will explain the recipe for obtaining the superpotential of a del Pezzo surface from the special Lagrangian fibration. Consider a sequence of pairs $(Y_i,D_i)$ converging to the large complex structure limit. Let $\check{X}_i$ be the hyperK\"ahler rotation of $X_i$ as in Theorem \ref{CJL main thm} and let $\check{Y}_i$ be the rational elliptic surface compactifying $\check{X}_i$. Notice that $\check{Y}_{\infty}$ is not defined.
 \begin{thm}[Lau-Lee-L., \cite{LLL}]
 	The limit $\lim_{i\rightarrow \infty}\check{Y}_i$ exists. Moreover, the limit as a complex manifold is the compactification of the superpotential of $Y_i$ as a monotone symplectic manifold.\footnote{Notice that the superpotential of $Y_i$ as a monotone symplectic manifold does not depend on the complex structures of $Y_i$.}
 \end{thm}
This result is perhaps an unexpected consequence of the existence of an SYZ fibration. Recall that in the early days of mirror symmetry it was thought that hyperK\"ahler rotation could be the underlying mechanism for mirror symmetry of Calabi-Yau surfaces, but this expectation was soon quickly shown to be inaccurate by the work of Gross-Wilson \cite{GW}. When $D_t$ is degenerating to a nodal curve, the SYZ fibration near infinity of $X_t$ is collapsing, and thus $(Y_0,D_0)$ can be viewed as the large complex structure limit from the SYZ point of view. With this interpretation, 
the theorem above roughly says that ``the limit of hyperK\"ahler rotation towards the large complex structure limit gives the mirrors for del Pezzo surfaces." 
Notice that the $Y_t$ are all symplectomoprhic for different $t$ and thus share the same mirror superpotential. The proof of the theorem is by direct computation of both sides and comparing them: both lead to the distinguished rational elliptic surface $\check{Y}_e$ in the deformation family. It would be interesting to give a direct Floer-theoretic interpretation of this result.

\subsection{Torelli Theorem for Gravitational Instantons of $ALH^*$}
 Gravitational instantons were introduced by Hawking as the building block for his Euclidean quantum gravity theory \cite{Haw}. Mathematically, gravitational instantons are non-compact complete hyperK\"ahler $4$-manifolds with $L^2$ curvature tensor. The first gravitational instantons were of type $ALE$, which are of maximal volume growth. Here $ALE$ is an abbreviation of {\em asymptotically locally Euclidean}. Later, new classes of gravitational instantons were discovered and classified into various types: e.g. $ALF, ALG, ALH$ having volume growth $r^3, r^2, r$ respectively. The first stands for {\em asymptotically locally flat} and the latter two are simply named by induction. Later, using rational elliptic surfaces, Hein \cite{Hein} constructed the new examples beyond the $ALE/F/G/H$ classification, which are now called $ALH^*$, and $ALG^*$. Precisely, Hein showed that, given any rational elliptic surface $\check{Y}$ and a fibre $\check{D}$, the complement $\check{X}=\check{Y}\setminus \check{D}$ always admits a complete hyperK\"ahler metric asymptotic to the semi-flat metric and with $L^2$ curvature. Moreover, the gravitational instanton $\check{X}$ is $ALH, ALG, ALG^*$, or  $ALH^*$ if $\check{D}$ is smooth, a singular fibre with finite monodormy, an $I_d^*$ singular fibre, or an $I_d$ singular fibre, respectively. The $ALH^*$-gravitational instantons, which are relevant for our discussion, have volume growth $r^{4/3}$. Although the $ALG^*$-gravitational instantons have volume growth $r^2$, they have different curvature decay than the $ALG$-gravitational instantons. When a gravitational instanton is not of maximal volume growth, Sun-Zhang \cite{SZ} found that its infinity is modeled by a certain Gibbons-Hawking ansatz, thereby concluding that all gravitational instantons are one of the above six types.
   
One may want to further classify the gravitational instantons among each type. This question is separated into three parts: 
 \begin{enumerate}
 	\item What are the diffeomorphism types of gravitational instantons of a fixed type? 
 	\item Among each diffeomorphism type, given three cohomology classes in $H^2$, can they be realized as the the cohomology class of a hyperK\"ahler triple of a gravitational instanton of the given type?
 	\item If there are two gravitational instantons of the same diffeomorphism type with the same cohomology classes of hyperK\"ahler triple, are the two gravitational instantons the same? 
 \end{enumerate}
  The set of all possible triples of second cohomology classes that can be realized as gravitational instantons is called the {\em period domain}. The third question asks if every point in the period domain parametrizes exactly one gravitational instanton and, if true, is usually called the {\em Torelli theorem} for gravitational instantons. 
  
 Kronheimer first completely classified the $ALE$ gravitational instantons \cite{Kro}: every $ALE$-gravitational instanton up to hyperK\"ahler rotation is biholomorphic to a crepant resolution of a quotient of $\mathbb{C}^2$ by a finite subgroup of $SU(2)$. Moreover, for any triple of second cohomology classes not vanishing simultaneously on any $(-2)$-classes, there exists exactly one $ALE$-gravitational instanton realizing it. In the case of $ALG, ALH, ALH^*, ALG^*$, the gravitational instantons (up to suitable hyperK\"ahler rotations) can be compactified to rational elliptic surfaces \cite{CC, CJL3, HSVZ2, CV}. These results answer the first question in these cases. The Torelli type theorems for $ALH$, $ALF$, $ALH^*$, $ALG$, $ALG^*$ were proven in \cite{CC,CC2, CJL3, CVZ, CVZ2}. For the rest of this section, we will explain the Torelli theorem for $ALH^*$-gravitational instantons as a quick application of the results discussed in this paper. 
 
 First we recall the Torelli theorem for K3 surfaces as an analogue: all the K3 surfaces are diffeomorphic \cite{Ko}, the cohomology class of the $(2,0)$-form determines the complex structure of the K3 surface \cite{SS, BR, LP}, and Yau's celebrated theorem for the Calabi conjecture \cite{Y} determines the Ricci-flat metric uniquely within a prescribed K\"ahler class. These results in order answer the questions above. We will follow the same strategy to prove the Torelli theorem of gravitational instantons.
  
  \begin{defn}
  	 A gravitational instanton $(X,\omega, \Omega)$ is of type $ALH^*$ if there exists a diffeomorphism $F:X_{\mathcal{C}}\rightarrow X$ such that   
  	    \begin{align} \label{ALH* defn}
  	       \parallel \nabla^k (\omega-(F^{-1})^*\omega_{\mathcal{C}})\parallel =O(r^{-k-\epsilon}),  \parallel \nabla^k (\Omega-(F^{-1})^*\Omega_{\mathcal{C}})\parallel =O(r^{-k-\epsilon}),
  	    \end{align} for some $\epsilon>0$, where $X_{\mathcal{C}},\omega_{\mathcal{C}}, \Omega_{\mathcal{C}}$ were defined in Section \ref{sec: Calabi ansatz}.
  \end{defn} The estimates \eqref{ALH* defn} is improved by Sun-Zhang \cite{SZ} to exponential decay. Thus the same argument as in Section \ref{sec: SYZ in dP} implies the following uniformization result:
 \begin{thm} \label{uniformization} Given an $ALH^*$-gravitational instanton $(X,\omega,\Omega)$, the 
  hyperK\"ahler rotation  $(\check{X},\check{\omega},\check{\Omega})$ given by \eqref{HK rel} can be compactified to a rational elliptic surface $\check{Y}$ by adding an $I_d$-fibre. In particular, there are exactly $10$ diffeomorphism types of $ALH^*$-gravitational instantons. 
 \end{thm}
  From the local model calculation, we have that the K\"ahler form $\check{\omega}$ is asymptotic to some (possibly non-standard) semi-flat metric. 

   Now assume that there are two $ALH^*$-gravitational instantons $(X_i,\omega_i,\Omega_i)$ with a diffeomorphism $F:X_2\cong X_1$ identifying the cohomology classes of the hyperK\"ahler triples. By Theorem \ref{uniformization}, suitable hyperK\"ahler rotations $\check{X}_i$ can be compactified to rational elliptic surfaces $\check{Y}_i$ by adding an $I_d$-fibre $\check{D}_i$. By using a Lemma of Friedman-Morgan \cite{FM} and computing the mapping class group of the Calabi model, the diffeomorphism $F:X_2\rightarrow X_1$ can be extended to $\check{Y}_2\rightarrow \check{Y}_1$ up to homotopy and  $F$ sends the homology classes of fibres/sections to those of fibres/sections. Then Theorem \ref{Torelli Looijenga} implies the isomorphism of the pairs $(\check{Y}_1,\check{D}_1)\cong (\check{Y}_2,\check{D}_2)$.\footnote{Caveat: due to the K\"ahler cone of $\check{X}_i$ being one dimension higher than the restriction from $\check{Y}_i$, it requires some treatment to argue the isomorphism actually preserves the K\"ahler classes.}

  Once the the pair $(\check{Y},\check{D})$ is determined by the Torelli theorem of Looijenga pairs, Theorem~\ref{uniqueness modulo aut} guarantees that the K\"ahler class $[\check{\omega}]$ will uniquely determine $\check{\omega}$ up to a translation by a holomorphic section (relative to a reference section). Notice that translation by a holomorphic section doesn't alter $\check{\Omega}$. Thus, we reach the Torelli theorem for $ALH^*$-gravitational instantons
  \begin{thm} \cite[Theorem 3.9]{CJL3} \label{Torelli ALH*}
  		Let $(X_i,\omega_i,\Omega_i)$ be $ALH^*$ gravitational instantons such that there exists a diffeomorphism $F:X_2\cong X_1$ with 
  	\begin{align*}
  	F^*[\omega_1]=[\omega_2]\in H^2(X_2,\mathbb{R}), F^*[\Omega_1]=[\Omega_2]\in H^2(X_2,\mathbb{C}).
  	\end{align*} Then there exists a diffeomorphism $f:X_2\rightarrow X_1$ such that $f^*\omega_1=\omega_2$ and $f^*\Omega_1=\Omega_2$. 
  \end{thm}
% \begin{rk}
% 	 The Torelli theorems for $ALG,ALG^*$-gravitational instantons is later proven by Chen-Viaclovsky-Zhang \cite{CVZ2}.
% \end{rk}

There are a few corollaries of this Torelli theorem worth mentioning. First of all, we partially answer a question raised by Tian-Yau \cite{TY} in the $2$-dimensional case.
 \begin{cor}
 	 Let $Y$ be a weak del Pezzo surface, $D$ be a smooth anti-canonical divisor and $X=Y\setminus D$. Fix a meromorphic $2$-form $\Omega$ on $Y$ with a simple pole along $D$. For each K\"ahler class of $X$, there exists a unique hyperK\"ahler metric $\omega$ in the given K\"ahler class with $L^2$ curvature. 
 \end{cor}
 \begin{proof}
 	The existence is guaranteed by \cite[Theorem 5.1]{TY}. Conversely, if there exists one such hyperK\"ahler metric $\omega$, then $(X,\omega,\Omega)$ is a gravitational instanton by \cite{SZ}. Since there are only finitely many possible diffeomorphism types of gravitational instantons, $(X,\omega,\Omega)$ must be an $ALH^*$-gravitational instanton. Then the corollary follows from Theorem \ref{Torelli ALH*}. 
 \end{proof}	
  In Theorem \ref{CJL main thm} and with the notation there, given any primitive class in $H_1(D,\mathbb{Z})$, there is a corresponding SYZ fibration in $X$. The monodromy of moduli of pairs $(Y,D)$ acts on $H_2(X,\mathbb{Z})$ and in particular induces an action on $\mbox{Im}(H_1(D,\mathbb{Z})\rightarrow H_2(X,\mathbb{Z}))\cong \mathbb{Z}^2$. A direct application of the Torelli theorem is to answer a question asked in \cite{HK}.
  \begin{prop}[Lau-Lee-L., \cite{LLL}]
  	 For any monodromy in $\mbox{Aut}(H_2(X,\mathbb{Z}))$, there exists an isometry of $X$ such that its induced map on $H_2(X,\mathbb{Z})$ realizes the given monodromy and sends special Lagrangian fibration to special Lagrangian fibraitons. In particular, the symplectic/complex affine structure of the different SYZ fibrations are isomorphic. 
  \end{prop}
  
  The last application is a different proof of another uniformization theorem of $ALH^*$-gravitational instantons due to Hein-Sun-Viaclovsky-Zhang. 
  \begin{thm}[Hein-Sun-Viaclovsky-Zhang, \cite{HSVZ2}] \label{another compactification}
  	Any $ALH^*$-gravitational instanton $(X,\omega,\Omega)$ can be compactified to a weak del Pezzo surface $Y$ by adding a smooth anti-canonical divisor $D$ such that $\Omega$ is a meromorphic $2$-form on $Y$ with a simple pole along $D$. 
 \end{thm}	
Before we give the proof, a dimension count for the moduli space of hyperK\"ahler triples from different geometries is a good sanity check. On the one hand, the moduli space of rational elliptic surfaces with an $I_d$-fibre is of complex dimension $9-d$, the meromorphic $2$-form with a simple pole on the given fibre is unique up to a $\mathbb{C}^*$-scaling, the K\"ahler cone of the complement is of real dimension $11-d$, and thus the total real dimension is $31-3d$. On the other hand, the moduli space of pairs of weak del Pezzo surfaces with a smooth anti-canonical divisor is of complex dimension $10-d$, the meromorphic volume with a simple pole along the given smooth anti-canonical divisor is unique up to $\mathbb{C}^*$-scaling, the K\"ahler cone of the complement is of real dimension $9-d$,\footnote{This is due to the long exact sequence of pairs analogous to \eqref{les} and the fact that the Calabi ansatz is exact near infinity.} and thus the total real dimension $31-3d$ again. 

 \begin{proof}[ Proof of Theorem \ref{another compactification}] 
 	 Consider the space $\mathscr{A}$ of $5$-tuples $(Y,D,[\omega],\alpha,c)$, where 
 	 \begin{enumerate}
 	 	\item $Y$ is a weak del Pezzo surface of degree $d$ in the deformation family of $Y_0$.
 	 	\item $D\in |-K_Y|$ is a smooth anti-canonical divisor, 
 	 	\item $[\omega]\in H^2(Y\setminus D,\mathbb{R})$,
 	 	\item $\alpha \in H_1(D,\mathbb{Z})$ is a primitive class, 
 	 	\item $c\in \mathbb{R}_+$.
 	 \end{enumerate} We identify two elements $(Y,D,[\omega],\alpha,c)$, $(Y',D',[\omega'],\alpha',c)$ if there exists an isomorphism of pairs $f:(Y,D)\rightarrow (Y',D')$ such that 
     \begin{align*}
       f_*\alpha=\alpha', f^*[\omega']=[\omega], c=c'.
     \end{align*}
     There exists a natural topology on $\mathscr{A}$. Similarly we will consider  $\mathscr{B}$ be the space of $4$-tuples of the form $(\check{Y},\check{D}, [\check{\omega}],\check{c})$, where 
     \begin{enumerate}
     	\item $\check{Y}$ is a rational elliptic surface,
     	\item $\check{D}$ is an $I_d$-fibre of $\check{Y}$,
     	\item $[\check{\omega}]\in H^2(\check{Y}\setminus \check{D},\mathbb{R})$ is a K\"ahler class, 
     	\item $\check{c}\in \mathbb{R}_+$. 
     \end{enumerate} We identify elements in $\mathscr{B}$ similarly, and there is a natural topology on $\mathscr{B}$. Given a $4$-tuple $(\check{Y},\check{D},[\check{\omega}],\check{c})$, there exists a unique meromorphic $2$-form $\check{\Omega}$ on $\check{Y}$ with a simple pole along $\check{D}$ normalized so that its integral over the bad cycle is $\check{c}$. For $\check{c}$ large enough, there exists a Ricci-flat metric $\check{\omega}$ in the cohomology class $[\check{\omega}]$ asymptotic to $\check{c}\check{\omega}_{sf}$ from the construction of Hein and unique up to translation by Theorem \ref{uniqueness modulo aut}. In particular, $2\check{\omega}^2=\check{\Omega}\wedge \bar{\check{\Omega}}$. Later from the proof below, one would see that the constraint of $\check{c}$ being large can be removed.  
 
 	Given an element $(Y,D,[\omega],\alpha,c)\in \mathscr{A}$, there exists a unique scaling of the Tian-Yau metric $\omega$ which is asymptotic to $c\omega_{\mathcal{C}}$ by \cite[Theorem 5.1]{TY}. From Theorem \ref{CJL main thm}, there exists an SYZ fibration in $X$ with fibre class corresponding to $\alpha$. The hyperK\"ahler rotation of $X$ via \eqref{HK rel} can be compactified to a rational elliptic surface $\check{Y}$ by adding an $I_d$-fibre $\check{D}$ with a Ricci-flat metric $\check{\omega}$ asymptotic to some scaling of a (possibly non-standard) semi-flat metric.  To sum up, there exists a continuous map 
	\[
	\mathscr{HK}:\mathscr{A}\rightarrow \mathscr{B}.
	\]
	We remark, after hyperK\"ahler rotation, the homology class in $H_2(X,\mathbb{Z})$ corresponding to $\beta\in H_1(D,\mathbb{Z})$ such that $\langle \alpha,\beta\rangle=1$ becomes precisely the bad cycle of $(\check{Y},\check{D})$ up to a multiple of the fibre. This information is helpful for the normalizing the $(2,0)$-form to apply Theorem \ref{Torelli Looijenga}. 
 	
 	We claim that the hyperK\"ahler rotation map $\mathscr{HK}$ is surjective. We will first prove the theorem assuming the claim.
 	Take an $ALH^*$-gravitational instanton $(X,\omega,\Omega)$ and a choice of homology class of special Lagrangian $[L]\in H_2(X,\mathbb{Z})$. By adjusting the phase of $\Omega$, we may assume that $\int_{[L]}\Omega>0$. Then Theorem \ref{uniformization} gives an element in $\mathscr{B}$, which coincides with \\$\mathscr{HK}(Y',D',\alpha',\beta',c')$. In particular, we have biholomorphisms 
 	 \begin{align*}
 	    \check{X}\cong (\underline{X},J), \check{X}\cong (\underline{X'},J').
 	 \end{align*} where $J$ (and $J'$) denote the complex structure determined by the $2$-form $\omega-i\mbox{Im}\Omega$ (and $\omega'-\mbox{Im}\Omega'$ respectively). Notice that there exists a unique meromorphic $2$-form on $\check{Y}$ with a simple pole along $\check{D}$ up to $\mathbb{C}^*$-scaling. 	 
 	 Using the normalization in the previous paragraph and the fixed scaling $\check{c}$, $[\Omega]=[\Omega']$ and $[\omega]=[\omega']$ via the biholomorphism from the surjectivity assumption.  	
 	Then from Theorem \ref{Torelli ALH*}, we have an isometry of gravitational instantons $(X',\omega',\Omega')\cong (X,\omega,\Omega)$. In particular, $X$ can be compactified to a compact complex surface which is biholomorphic to $Y'$. 
 	
 	To prove the claim, it suffices to prove that the image of $\mathscr{HK}$ is both open and closed. We will first prove the openness. Assume $\mathscr{HK}(Y,D,[\omega],\alpha,c)=(\check{Y},\check{D},[\check{\omega}],\check{c})$. Recall that the $j$-invariant defines a non-constant holomorphic function on the moduli space of pairs of weak del Pezzo surfaces with smooth anti-canonical divisors. In particular, it is an open map. Together with \cite[Proposition 3.12]{Fr}, we may achieve all possible nearby period of $\Omega$ by deforming the pair $(Y,D)$, i.e. deforming $D$ and deforming the blow-up loci on $\mathbb{P}^2$. From \cite[Theorem 5.1]{TY}, there exists a Ricci-flat metric asymptotic to the Calabi ansatz in any K\"ahler class $[\omega']$ close enough to $[\omega]$ on a $X'$,\footnote{Here we identify the cohomology class via parallel transport in the deformation family.} where $X'=Y'\setminus D'$ and $(Y',D')$ is a deformation of $(Y,D)$. This implies that after hyperK\"ahler rotation, one can achieve all small deformations of periods of $[\check{\Omega}]$ and $[\check{\omega}]$, where $\check{\Omega}$ is the meromorphic $2$-form on $\check{Y}$ with a simple pole along $\check{D}$. Thus, $\mathscr{HK}$ is open by Theorem \ref{Torelli Looijenga}. Finally, we will prove that $\mathscr{HK}$ is closed. Assume that there is a convergent sequence in $\mathscr{B}$,
 	 \begin{align}\label{convergence}
 	 \mathscr{HK}(Y_i,D_i,[\omega_i],\alpha_i,c_i)=(\check{Y}_i,\check{D}_i,[\check{\omega}_i],\check{c}_i)\rightarrow (\check{Y},\check{D},[\check{\omega}],\check{c}). 
 	 \end{align} Notice that the modulus of $D_i$ can be recovered from the integral of\\ $[\check{\omega}_i-i\mbox{Im}\check{\Omega}_i]$ over the elliptic fibre and the bad cycle of $\check{Y}_i$. From the convergence \eqref{convergence}, we have $D_i\rightarrow D$ for some smooth elliptic curve $D$. The convergence of periods of $[\check{\omega}_i-i\mbox{Im}\check{\Omega}_i]$ on $H_2(\check{X}_i,\mathbb{Z})$ implies there exists a weak del Pezzo surface $Y$ such that $(Y_i,D_i)\rightarrow (Y,D)$ by Theorem \cite[Theorem 6.4]{McM}. Then $[\omega]=[\mbox{Re}\check{\Omega}]$ up to parallel transport, and
 	 $\alpha$ is determined by $\alpha_i$ via parallel transport, and $c=\lim_i c_i$. It is not hard to see that $\mathscr{HK}(Y,D,[\omega],\alpha,c)=(\check{Y},\check{D},[\check{\omega}],\check{c})$, so $\mathscr{HK}$ is closed.

 \end{proof}	
  There are several questions arising naturally from the Torelli theorem. Recall that given an $ALH^*$-gravitational instanton $X$ with a diffeomorphism to a fixed reference $ALH^*$-gravitational instanton $X_0$, the hyperK\"ahler triple of $X$ induces an element in $H^2(X_0,\mathbb{R})^3$ via the diffeomorphism. The map from the space of the $ALH^*$-gravitational instantons to $H^2(X_0,\mathbb{R})^3$ is the analogue of the period map of K3 surfaces and so $H^2(X_0,\mathbb{R}^3)$ is the analogue of the period domain. Theorem \ref{Torelli ALH*} implies the surjectivity of the period map. To understand the space of $ALH^*$-gravitational instantons, one would want to characterize the image of the period domain. This will be studied in the upcoming work \cite{LL}.

  Another question comes from the following observation:
  the different compactifications of $ALH^*$-gravitational instantons after hyperK\"ahler rotations motivate the possibility of different compactifications of other types of gravitational instanton. It has been proven that $ALG$ and $ALG^*$-gravitational instantons can be compactified to rational elliptic surfaces after suitable hyperK\"ahler rotation \cite{CC}\cite{CV}. 
  \begin{conj}
  	 Do $ALG$ or $ALG^*$-gravitational instanton (up to suitable hyperK\"ahler rotation) admit compactifications (obtained by adding a canonical cycle) to rational surfaces other than the rational elliptic surface?
     \end{conj}
     For instance, the conjecture holds for the $ALG$-gravitational instantons after suitable hyperK\"ahler rotations \cite{CL}. 

\subsection{Sympectomorphisms}
  Let $Y$ be a del Pezzo surface of degree $d$, $D\subseteq Y$ be a smooth anti-canonical divisor, and $(\check{Y},\check{D})$ be the rational elliptic surface pair in Theorem \ref{CJL main thm}. Then from \eqref{HK rel}, any automorphism of the pair $(\check{Y},\check{D})$ that preserves the meromorphic volume form of $\check{Y}$ with a simple pole along $\check{D}$ gives rise to a symplectomorphism of $X=Y\setminus D$ with respect to the Tian-Yau metric \cite[Proposition 5.10]{CJL2}.  
  Moreover, we have the following observation as a direct corollary of Theorem \ref{uniqueness thm}: 
   \begin{prop}\label{HK sym}
   	In the above setting, if $\phi:(\check{Y},\check{D})\cong (\check{Y},\check{D})$ is an isomorphism of pairs such that
   	 \begin{enumerate}
   	 	\item 	 $\phi$ preserves the meromorphic volume form of $\check{Y}$ with simple pole along $\check{D}$ and 
   	 	\item $\phi$ preserves the K\"ahler class of $\check{X}=\check{Y}\setminus \check{D}$ from  the hyperK\"ahler rotation, 
   	 \end{enumerate}
   	then the same underlying map before the hyperK\"ahler rotation, still denoted by $\phi$, is an automorphism of the pair $(Y,D)$. 
   \end{prop}
  In particular, we recover a classical theorem in the case $Y=\mathbb{P}^2$.
  \begin{thm}
  	Given any smooth cubic $D$ in $\mathbb{P}^2$, there exists $\mathbb{Z}_3\oplus \mathbb{Z}_3\subseteq \mbox{Aut}(\mathbb{P}^2)$ preserving $D$. 
  \end{thm}
  \begin{rk}
  	 Generalizations of Proposition \ref{HK sym} to Looijenga pairs via mirror symmetry were proven in \cite{HK, KW}. 
  \end{rk}

\end{document}